\documentclass[12pt]{article}
\usepackage[utf8]{inputenc}
\usepackage{geometry}
\usepackage[colorlinks]{hyperref}

\usepackage{graphicx}
\usepackage{tabularx}

\usepackage[bb=boondox]{mathalfa}  

\usepackage{amssymb,amsmath,amsthm}
\usepackage{mathtools}
\usepackage{mathrsfs}  
\usepackage{fancyhdr}
\usepackage{enumitem}
\usepackage{bbm}
\usepackage{dsfont}
\usepackage{pgfplots}
\usepackage{yhmath}
\usepackage{ragged2e}
\usepackage{float}
\usepackage{extarrows}
\usepackage{marginnote}
\pgfplotsset{compat=1.18}

\newcommand{\N}{\mathbb{N}}
\newcommand{\Z}{\mathbb{Z}}
\newcommand{\R}{\mathbb{R}}

\newcommand{\E}{\mathbb{E}}

\newcommand{\Prob}{\mathbb{P}}

\numberwithin{equation}{section}

\newtheorem{theorem}{Theorem}[section] 
\newtheorem{corollary}[theorem]{Corollary}

\newtheorem{lemma}[theorem]{Lemma}

\theoremstyle{definition}

\theoremstyle{remark}
\newtheorem{remark}[theorem]{Remark}

\title{Ergodicity of the voter model with dynamic anti-voter bonds}
 
\author{Jhon Astoquillca \thanks{Institute of Mathematics, Statistics and Computer Science, University of S\~ao Paulo} \and Daniel Valesin \thanks{Department of Statistics, University of Warwick} }
\date{\today}

\begin{document}

\maketitle

\begin{abstract}
The voter model with anti-voter bonds is a variant of the classical voter model in which the edges of the underlying graph are assigned signs. At each update, a voter chooses a neighbour according to a transition kernel; interactions across a positive edge follow the usual voter dynamics, so that a site adopts the current opinion of its chosen neighbour, whereas interactions across a negative edge lead to the adoption of the opposite opinion.

In this work, we introduce a new variant in which the edge signs evolve dynamically according to dynamical percolation with density parameter~$p \in (0,1)$ and speed~$\mathsf v \in (0,\infty)$, where the two states of the process represent positive and negative edges. This defines a joint spin–bond Markov process. Following Liggett’s notion of ergodicity, we prove that this process is ergodic on any simple graph with countably many vertices, with an arbitrary transition kernel of adoption rates and for all choices of the parameters of the edge dynamics.
\end{abstract}



\section{Introduction}
\textbf{Ergodicity.} One of the central objectives in the study of interacting particle systems is to understand their equilibrium structure and long-time behavior. Two fundamental problems naturally emerge in this context. The first is the description of the set of stationary probability measures of the system. Since the state space of such models is typically a compact metric space, the Krein–Milman theorem ensures that this set is determined by its extremal elements, that is, by those stationary measures that cannot be written as non-trivial convex combinations of other stationary measures.

The second fundamental problem is to determine the domain of attraction of each extremal stationary measure. More precisely, given an extremal stationary measure~$\mu$, one seeks to characterize all probability measures~$\nu$ such that the law of the process started from an initial configuration sampled according to~$\nu$ converges weakly to~$\mu$ as~$t \to \infty$.

In this context, we focus on the situation where the system exhibits a form of ``ergodicity'', namely, convergence to a unique stationary state independently of the initial configuration. There is no universal consensus on the precise meaning of ergodicity for interacting particle systems, and terminology varies across the literature. In this work, we follow the convention of Liggett~\cite[Definition 1.9]{Liggett2005} and say that a Markov process on a compact metric space is \emph{ergodic} if
\begin{itemize}
\item it admits a unique stationary measure~$\mu$; and
\item for any probability measure~$\nu$ the law of the process started from~$\nu$ converges weakly to~$\mu$ as~$t \to \infty$.  
\end{itemize} 
\noindent \textbf{The voter model.} In this work, we study a variant of the standard voter model. The voter model is a classical interacting particle system that serves as a simple model for the collective behavior of individuals who continuously update their opinions through local interactions. It was independently introduced by Clifford and Sudbury~\cite{CliffordSudbury73} and by Holley and Liggett~\cite{HolleyLiggett75}, and is treated in detail in the expository texts~\cite{Liggett2005,Swart2017}.

Numerous variations of the voter model have been studied on general graphs. We focus on the voter model with anti-voter bonds dynamics. In this model, the edges of the underlying graph are assigned signs. Interactions across a positive edge follow the usual voter dynamics, whereas interactions across a negative edge lead to the adoption of the opposite opinion. This variant is known to exhibit temporal ergodicity under suitable conditions. We refer to Section~\ref{ss_historical_background} for a discussion of the existing results.

In this work, we introduce a new variant in which the signs of the edges evolve according to a dynamical percolation process. Our main objective is to determine whether the resulting joint spin-bond system is ergodic in the sense of Liggett. We now briefly describe the model in order to state our main result; a formal definition is provided in Section~\ref{ss_the_model}. \\

\noindent \textbf{The voter model with dynamic anti-voter bonds.} Consider a simple graph~$G = (V,E)$ with countable vertex set, and a matrix~$Q$ with entries~$Q = \{q(x,y): x,y \in V\}$ with non-negative off-diagonal entries, satisfying
\begin{equation}\label{eq_adoption_rate_introduction}
q(x,x) = -\sum_{y \in V}q(x,y) = -1 \quad \text{for all } x \in V    
\end{equation}
and also~$q(x,y) = 0$ if~$\{x,y\} \notin E$. We will refer to such a matrix as a \emph{$Q$-matrix} for the graph. Given~$\mathsf v \in (0,\infty)$ and~$p \in (0,1)$, the~\emph{voter model with dynamic anti-voter bonds} on~$G$ with \emph{adoption rate}~$Q$ is a Markov process on~$\{-1,1\}^{V \cup E}$, denoted by~$(\eta_t,\zeta_t)_{t \ge 0}$, where~$\eta_t(x)$ represents the opinion of the voter at~$x$ and~$\zeta_t(e)$ the sign of the bond~$e$, both at time~$t$. We describe the dynamics of the process as follows. 

For~$(\eta,\zeta) \in \{-1,+1\}^{V \cup E}$ and~$(x,e) \in V\times E$, the following jumps can occur for the state of~$(\eta(x),\zeta(e))$, with rates as specified (where~$\mathds 1A$ denotes the indicator function of the event~$A$):
\begin{align*}
(+1,\zeta(e) ) \mapsto (-1,\zeta(e)) \quad & \text{at rate } \sum_{y:y \sim x} q(x,y) \cdot 1 \{ \eta(y) \cdot \zeta(\{x,y\}) = -1 \} \\[2mm]
(-1,\zeta(e) ) \mapsto (+1,\zeta(e)) \quad & \text{at rate } \sum_{y:y \sim x} q(x,y) \cdot 1 \{ \eta(y) \cdot \zeta(\{x,y\}) = +1 \} \\[2mm]
(\eta(x),-1 ) \mapsto (\eta(x),+1) \quad & \text{at rate } p \cdot \mathsf{v} \\[2mm]
(\eta(x),+1 ) \mapsto (\eta(x),-1) \quad & \text{at rate } (1-p) \cdot \mathsf{v}
\end{align*}
In words, the bond process~$\zeta_t(e)$ evolves as a two-state Markov chain independently for each edge~$e$. In the site process, each site (voter)~$x$ at rate~1 chooses a neighbour with probability according to~$\{q(x,y): y \sim x\}$ and adopts the opinion (resp. the opposite opinion) of~$y$ if the edge~$\{x,y\}$ is positive (resp. negative) at that time.

We now state our main result.
\begin{theorem}\label{thm_main_ergo}
The voter model with dynamical anti-voter bonds defined on an arbitrary graph~$G$, with adoption rates given by a~$Q$-matrix satisfying~\eqref{eq_adoption_rate_introduction} and with environment parameters~$p \in (0,1)$ and~$\mathsf v \in (0,\infty)$ is ergodic.    
\end{theorem}

\subsection{Background: Static scenario}\label{ss_historical_background}
We now briefly review existing results on ergodicity for the voter model with positive and negative bonds fixed over time.

Given a graph~$G$ and a~$Q$-matrix as above, and a configuration~$\zeta \in \{+1,-1\}^E$ assigning positive and negative signs to the edges of~$G$,~\emph{the voter model with anti-voter bonds} is the Markov process~$(\eta_t)_{t \ge 0}$ on~$\{-1,+1\}^V$  whose dynamics are defined analogously to the dynamical version. In this case, however, the signs of the edges are fixed by~$\zeta$ and remain constant throughout the evolution of opinions of the system.

\begin{enumerate}
    \item When~$\zeta$ is the all-positive configuration, that is,~$\zeta(e) = +1$ for all~$e \in E$, we recover the standard voter model. See~\cite{Liggett2005} for an extended description of its stationary measures. In particular, the consensus configurations, in which all voters share the same opinion, are stationary for the model. Hence, the process is not ergodic on any graph~$G$ for any choice of~$Q$-matrix.
    \item When~$\zeta$ is the all-negative configuration, that is,~$\zeta(e) = -1$ for all~$e \in E$, we obtain the~\emph{dissonant voter model or anti-voter model}. This model was introduced in~\cite{Matloff1977}. Consider the Markov chain on~$V \times \{-1,+1\}$ with~$\tilde Q$-matrix defined by
    $$ \tilde Q( (x,a);(y,b) ) = \left\{ \begin{array}{c}
    Q(x,y) \quad \text{if } a \cdot b = +1; \\[2mm]
    0 \hspace{9mm} \text{ otherwise.}
    \end{array} \right. $$
    It was shown in~\cite[Theorem 3.6]{Matloff1977} that the model on any graph~$G$ and~$Q$-matrix is ergodic if and only if the set of functions~$f:V \times \{-1,+1\} \to \R$ that are harmonic with respect to~$\tilde Q$ and satisfy
    $$f(x,-1) + f(x,+1) = 1 \quad \text{ for all } x \in V,$$ 
    are the constant function~$1/2$.
    \item When~$\zeta$ is neither entirely positive nor entirely negative, the authors in~\cite[Theorem 6.1]{Gantertetall2005} showed that for any locally finite and connected graph~$G$ with transition rates~$Q(x,y) = (\deg(x))^{-1} \cdot \mathds{1}\{x \sim y\}$, the model is ergodic if and only if
    $$ \lim_{t \to \infty} \mathbb P( \eta_t(x) = +1 ) = 1/2 \quad \text{for all } x \in V \text{ and } \eta_0 \in \{-1,+1\}^V. $$
    In~\cite[Theorem 1.4]{MaillardMountford2013}, it was further proved that when the model admits a unique stationary probability measure, then it is ergodic.
    \item A random configuration~$\zeta$ of signed bonds was studied in~\cite{Saada95} for~$G=(\Z^d,E(\Z^d))$ and~$Q(x,y) = 1/(2d) \cdot \mathds{1}\{x \sim y\}$. The configuration~$\zeta$ is given by a Bernoulli product measure on~$E(\mathbb{Z}^d)$ with density~$p$. It is shown in Theorem~3 of that article that, for every~$p \in (0,1)$, the voter model with anti-voter bonds determined by~$\zeta$ is almost surely ergodic.
\end{enumerate}



\section{The voter model with dynamic anti-voter bonds}
All graphs considered in this text have non-empty, countable vertex sets and are simple, that is, undirected graphs with no loops and at most one edge between any pair of vertices. Given a graph~$G = (V,E)$, we write~$x \sim y$ when~$\{x,y\} \in E$, and we denote by~$\deg(x)$ the degree of the vertex~$x$. We denote by~$\Omega_\mathrm{site} = \{-1,+1\}^V$ and~$\Omega_\mathrm{edge} = \{-1,+1\}^E$ the space of site and edge configurations, respectively. We then let~$\Omega = \Omega_\mathrm{site} \times \Omega_\mathrm{edge}$ be the joint configuration space.


\subsection{The model}\label{ss_the_model}
In this section, we define a dynamical version of the voter model with anti-voter bonds introduced in~\cite{Gantertetall2005,Matloff1977} by allowing the edges evolve as a dynamical percolation process. \\

\noindent \textbf{Dynamical percolation.} The basic bond percolation model on a graph~$G = (V,E)$ is obtained by fixing $p \in [0,1]$ and sampling the full configuration of edges from 
$$\pi^\mathrm{edge}_p := \bigotimes_{e \in E} \big( p \cdot \delta_{ \{+1\} } + (1-p) \cdot \delta_{ \{-1\} } \big) .$$
In this model, each edge has two possible states,~$+1$ and~$-1$, which encode the sign of the bond. This should be contrasted with the usual percolation framework, where edge states indicate whether an edge is present or absent; we do not follow that interpretation here.

Dynamical percolation is a dynamical variant of this model. It is defined as follows. Given~$p \in [0,1]$ and~$\mathsf v \in (0, \infty)$, each edge in $E$ evolves according to an independent two-state (positive and negative) continuous-time Markov chain with state space~$\{+1,-1\}$ and flips from negative to positive (resp. positive to negative) at rate~$\mathsf v \cdot p$ (resp.~$\mathsf v \cdot (1-p)$).  

The entire configuration of positive and negative edges at time~$t$, denoted by $\zeta_t$, is an element of~$\Omega_\mathrm{edge}$. The dynamics of this Markov process can be encoded in the pre-generator 
$$\mathcal{L}_\mathrm{dp}f(\zeta) = \sum_{e \in E }c(e,\zeta,\mathsf{v},p) \cdot [ f(\zeta^e) - f(\zeta) ],$$
where $f:\Omega_\mathrm{edge} \to \R$ is any function that only depends on finitely many coordinates, and for any~$\zeta~\in~\Omega_\mathrm{edge}$,
$$ \zeta^e(e') = \left\{ \begin{array}{ll}
    \zeta(e') & \text{if } e' \neq e \\[2mm]
    -\zeta(e) & \text{if } e' = e
\end{array} \right. \quad \text{and} \quad c(e,\zeta,\mathsf{v},p) = \left\{ \begin{array}{ll}
    \mathsf{v} \cdot p, & \text{if } \zeta(e)=-1 \\[2mm]
    \mathsf{v} \cdot (1-p), & \text{if } \zeta(e)=+1
\end{array} \right. $$
This process is stationary with respect to $\pi^\mathrm{edge}_p$. \\

\noindent \textbf{The voter model with anti-voter bonds.} Let us consider a graph~$G = (V,E)$ and a function~$Q:V \times V \to \R$ that satisfies
\begin{equation}\label{Q_def_adoption_rw}
\sum_{y \in V}Q(x,y) = 0, \quad Q(x,y) \left\{ \begin{array}{cc}
    \ge 0 & \text{if } x \sim y \\[2mm]
    = 0 & \text{if } x \not \sim y
\end{array} \right.  \; \text{ and } \; Q(x,x) = -1 \; \text{ for all } x \in V.   
\end{equation}
\emph{The voter model with dynamic anti-voter bonds} on~$G$ with \emph{adoption rate}~$Q$ and environment parameters~$p \in [0,1]$ and~$\mathsf{v} \in (0,\infty)$, denoted by~$(\eta_t,\zeta_t)$,~$t \geq 0$, is the Feller process with state space~$\Omega$ and pre-generator
\begin{equation}\label{Markov_pre_dyn}
\mathcal{L}_\mathrm{dyn} f(\eta,\zeta) = \sum_{(x,y) \in V^2: \; x\sim y} Q(x,y) \cdot \big[ f(\eta^{y \to x}_\zeta,\zeta) - f(\eta,\zeta)\big] + \mathcal L_\mathrm{dp}f_\eta(\zeta),
\end{equation}
where $f:\Omega \to \R$ is any function that only depends on finitely many coordinates, the function~$f_\eta:\Omega_\mathrm{edge} \to \R$ is obtained from~$f$ by fixing the first coordinate at~$\eta \in \Omega_\mathrm{site}$, and for any~$(\eta,\zeta) \in \Omega$, 
$$\eta^{y \to x}_\zeta(z) = \left\{ \begin{array}{ll}
    \eta(z) & \text{if } z \neq x;  \\[2mm]
    \eta(y)\cdot \zeta(\{x,y\}) & \text{otherwise} \end{array} \right. $$

For convenience in the calculations, we identify the two possible voter opinions as the states~$+1$ and~$-1$. Unlike the voter model on dynamical percolation introduced in~\cite{Astoquillca26}, no opinion adoption is rejected. If at time~$t$ a voter~$x \in V$ attempts to adopt the opinion of~$y \in V$, it will adopt the actual (resp. opposite) opinion of~$y$ if the edge~$\{x,y\}$ is positive (resp. negative) at that time.

\subsection{Duality}
In the literature on the voter model with anti-voter bonds, random walks with jumping rate according to~$Q$ from~\eqref{Q_def_adoption_rw}, paired with a process that tracks the negative bonds crossed by the walker, play a crucial role in describing the stationary measures of the model.

In this section, we define an analogous process in the setting of a dynamical environment. Throughout this section, we assume that the environment parameters~$p \in (0,1)$ and~$\mathsf v \in (0,\infty)$ are fixed. \\

\noindent \textbf{The sign process and coalescing random walks on dynamical signed edges.} To define the dual process on a graph~$G=(V,E),$ we introduce some notation. Given $z,x,y \in V$ and~$\mathsf i \in \{-1,+1\}$, we define 
\begin{equation*}
     z^{x \to y} := \left\{ \begin{array}{cc}
        y & \text{if } z = x, \\[2mm]
        z & \text{otherwise}
    \end{array} \right. \quad \text{ and } \quad \mathsf  i^{x,z} := \left\{ \begin{array}{cc}
        -\mathsf i & \text{if } z = x,  \\[2mm]
        \mathsf i & \text{otherwise.}
    \end{array} \right. 
\end{equation*}
In words, we update the base from~$z$ and~$\mathsf i$ to~$y$ and~$-\mathsf i$, respectively, only if~$z = x$. Moreover, for any integer~$k \in \N$ and vectors $\mathbf{x} = (x_1,\dots, x_k) \in V^k$ and~$\mathbf i = (\mathsf i_1,\dots, \mathsf i_k) \in \{-1,+1\}^k$ we write 
$$ \mathbf{x}^{x \to y} := (x^{x \to y}_1, \dots, x^{x \to y}_k) \quad \text{and} \quad \mathbf i^{x,\mathbf x} := (\mathsf i^{x,x_1}_1, \dots, \mathsf i^{x,x_k}_k ). $$
Finally, we denote by~$\mathcal{P}_\mathrm{fin}(E)$ the collection of finite subsets of $E$ and write
$$ \mathcal P_\mathrm{fin-disj}(E) := \{ (A,B) \in \mathcal P_\mathrm{fin}(E) \times \mathcal P_\mathrm{fin}(E) : A \cap B = \varnothing \}. $$

We are now ready to define the \emph{dual process}. Fix~$k \in \N$, and consider the continuous-time Markov chain 
$$ (Y^1_t, \dots, Y^k_t,\mathsf i^1_t,\dots, \mathsf i^k_t, A_t,B_t)_{t \ge 0} \quad \text{with state space} \quad \mathcal S_k := V^k \times \{-1,+1\}^k \times \mathcal P_\mathrm{fin-disj}(E)$$ 
and rate matrix~$Q_\mathrm{crw} = Q_\mathrm{crw}(p,\mathsf v,k,Q)$, defined as follows. Let~$(\mathbf y, \mathbf i, A,B) \in \mathcal S_k$ and~$x$ an entry of the vector~$\mathbf y \in V^k$,
\begin{equation}\label{matrix_q_crw_dyn}
    \begin{array}{ll}
       Q_\mathrm{crw} \big( ( \mathbf y,\mathbf i,A,B), ( \mathbf y^{x \to y},\mathbf i,A,B ) \big) = Q(x,y), & \{x,y\} \in A, \; y \in V;  \\[0.2cm]
       Q_\mathrm{crw} \big( ( \mathbf y,\mathbf i,A,B), ( \mathbf y^{x \to y},\mathbf i,A \cup \{x,y\},B ) \big) = p \cdot Q(x,y), & \{x,y\} \in (A \cup B)^c, \; y \in V;
       \\[0.2cm] 
       Q_\mathrm{crw} \big( ( \mathbf y,\mathbf i,A,B), ( \mathbf y^{x \to y},\mathbf i^{x,\mathbf y},A,B ) \big) = Q(x,y), & \{x,y\} \in B, \; y \in V; \\[0,2cm]
       Q_\mathrm{crw} \big( ( \mathbf y,\mathbf i,A,B), ( \mathbf y^{x \to y},\mathbf i^{x,\mathbf y},A,B \cup \{x,y\} ) \big) = (1-p) \cdot Q(x,y), & \{x,y\} \in (A \cup B)^c, \; y \in V; \\[0.2cm]
       Q_\mathrm{crw} \big( ( \mathbf y,\mathbf i,A,B), ( \mathbf y,\mathbf i,A \setminus \{e\},B \setminus \{e\} ) \big) = \mathsf v, & e \in A \cup B.
    \end{array}
\end{equation}
In words, the matrix rate~$Q_\mathrm{crw}$ shows that
\begin{enumerate}
    \item the process~$\mathbf Y_t := (Y^1_t,\dots,Y^k_t), \;t \ge 0$, evolves as a system of coalescing random walks with each walker jumping according to~$Q$. In contrast to other models of random walks on dynamical percolation~\cite{PeresStaufferSteif2015,Andresetall,Astoquillca26}, in this case, no jump is rejected;
    \item the environment process~$(A_t,B_t),\; t \ge 0,$ takes values in~$\mathcal P_\mathrm{fin-disj}(E)$, the pair~$(A_t,B_t)$ represents the set of positive and negative edges, respectively, known to the system of coalescing random walks at time~$t$. The remaining edges are considered unknown. When a random walk traverses such an unknown edge, the status of the edge is revealed by sampling a Bernoulli random variable with parameter~$p$: it is declared positive with probability~$p$, and negative with probability~$1-p$. 
    Note that the environment process can be started from some initial knowledge of the environment~$(A_0,B_0) \in \mathcal{P}_\mathrm{fin-disj}(E)$. A closely related process is defined in Section~4.2 of~\cite{Astoquillca26} in the setting where edges evolve between absent and present states.
    \item the sign process~$\mathbf i_t := (\mathsf i^1_t,\dots, \mathsf i^k_t), \;t \ge 0,$ takes values in~$\{-1,+1\}^k$. Each coordinate is associated to a random walk in~$(\mathbf Y_t)$ and flips its sign whenever the corresponding random walk crosses a negative edge. Its relevance, particularly for readers unfamiliar with the static model~\cite{Matloff1977, Saada95, Gantertetall2005, MaillardMountford2013}, will become clear in the discussion of the duality relation below. Analogously to the environment process, the sign process may start with either a positive or a negative sign.  
\end{enumerate}

We now prove that the Markov chain~$(\mathbf Y_t,\mathbf i_t, A_t,B_t)_{t \ge 0}$ and the voter model with dynamic anti-voter bonds satisfy a duality relationship. In order to state such a relationship, we add a bit more of notation. 

Given~$(\mathbf{x},\mathbf{i},A,B) \in \mathcal S_k$, we denote by~$P_{\mathbf{x},\mathbf{i},A,B}$ a probability measure under which~$(\mathbf Y_t,\mathbf i_t,A_t,B_t)_{t \ge 0}$ is defined and satisfies~$P_{\mathbf{x},\mathbf{i},A,B}( (\mathbf Y_0,\mathbf i_0,A_0,B_0) = (\mathbf{x}, \mathbf{i}, A ,B) ) = 1$. The expectation operator associated with this measure is denoted by~$E_{\mathbf x, \mathbf i, A,B}$.

Given~$(\eta,\zeta) \in \Omega$, we denote by~$P_{\eta,\zeta}$ a probability measure under which~$(\eta_t,\zeta_t)_{t \ge 0}$ is defined and satisfies~$P_{\eta,\zeta}( (\eta_0,\zeta_0) = (\eta, \zeta)) = 1$. The expectation operator associated with this measure is denoted by~$E_{\eta,\zeta}$.

\begin{lemma}\label{lemma_duality_raw}
Fix an arbitrary $k \in \N$, a graph~$G$ and a function~$Q$ satisfying~\eqref{Q_def_adoption_rw}. Then, the voter model with dynamic anti-voter bonds $(\eta_t,\zeta_t)_{t \ge 0}$ with adoption rate~$Q$ and environment parameters~$p$ and~$\mathsf v$, and the Markov chain~$(\mathbf Y_t,\mathbf i_t,A_t,B_t)_{t \ge 0}$ with matrix rate~$Q_\mathrm{crw}(p,\mathsf v,k,Q)$ are dual with respect to the function~$D = D_k:\Omega \times \mathcal S_k \to \R$ defined as
$$D(\eta,\zeta,\mathbf x,\mathbf i,A,B) = \varphi_\mathrm{site}(\eta,\mathbf x, \mathbf i)  \cdot \varphi_\mathrm{edge}(\zeta,A,B),$$
    where
    \begin{equation}\label{psi_varphi_def}
        \begin{aligned}
        & \varphi_\mathrm{site}(\eta,\mathbf x, \mathbf i) = \mathds{1}\{ \eta(x_1) = \mathsf i_1, \dots, \eta(x_k) = \mathsf i_k \} \; \text{ and} \\[2mm]
        & \varphi_\mathrm{edge}(\zeta,A,B) = p^{-|A|}(1-p)^{-|B|} \cdot \mathds{1}\{\zeta \equiv 1 \text{ on } A,\; \zeta \equiv -1 \text{ on } B\},
        \end{aligned}
    \end{equation}
    that is, for any~$(\eta,\zeta,\mathbf{x},\mathbf i,A,B) \in \Omega \times \mathcal S_k$ we have that
    \begin{equation}\label{duality_equality_raw}
        E_{\eta,\zeta}[D(\eta_t,\zeta_t,\mathbf{x},\mathbf{i},A,B)] = E_{\mathbf{x},\mathbf{i},A,B}[D(\eta,\zeta,\mathbf Y_t, \mathbf i_t,A_t,B_t)].
    \end{equation}
\end{lemma}

\begin{proof}
This proof is similar in spirit to the proof of duality in Lemma~4.11 of~\cite{Astoquillca26}. It is included here for completeness.

We write~$D_{\eta,\zeta}(\mathbf{x},\mathbf{i},A,B) = D_{\mathbf{x},\mathbf{i},A,B}(\eta,\zeta) =D(\eta,\zeta,\mathbf{x},\mathbf{i},A,B)$. Theorem~3.42 of~\cite{Liggett2010} states that a sufficient condition for~\eqref{duality_equality_raw} is to prove that for all~$(\eta,\zeta,\mathbf x, \mathbf i, A,B) \in \Omega \times \mathcal S_k$,
\begin{equation}\label{duality_to_proof}
         \mathcal L_\mathrm{dyn} D_{\mathbf{x},\mathbf{i},A,B}(\eta,\zeta) = \mathcal L_\mathrm{coal-dyn}D_{\eta,\zeta}(\mathbf{x},\mathbf{i},A,B),
\end{equation}
where~$\mathcal L_\mathrm{dyn}$ is the Markov pre-generator~\eqref{Markov_pre_dyn} and~$\mathcal L_\mathrm{coal-dyn}$ is the generator of the Markov chain~\eqref{matrix_q_crw_dyn}.
    
Fix a configuration~$(\eta,\zeta, \mathbf{x}, \mathbf{i},A,B) \in \Omega\times \mathcal S_k$ and sites~$x,y \in V$ such that~$x \sim y$. Note that for any~$z \in V$,
\begin{equation}\label{eta_zeta_eq}
    \eta^{y \to x}(z) = \eta(z^{x \to y}).
\end{equation}
We consider the following three cases associated with the edge~$\{x,y\} \in E$. \\[1mm]

\noindent \textbf{Positive edge}, that is,~$\{x,y\} \in A$. In either case~$\zeta(\{x,y\})=+1$ or~$\varphi_\mathrm{edge}(\zeta,A,B) = 0$, using~\eqref{eta_zeta_eq}, we have that
    \begin{equation}\label{duality_proof_1}
        D_{\mathbf{x}, \mathbf{i}, A ,B} (\eta^{y \to x}_\zeta, \zeta) - D_{\mathbf{x}, \mathbf{i}, A ,B}(\eta, \zeta) = D_{\eta,\zeta}( \mathbf{x}^{x \to y}, \mathbf{i}, A,B ) - D_{\eta,\zeta}( \mathbf{x}, \mathbf{i}, A,B );
    \end{equation} 

\noindent \textbf{Negative edge}, that is,~$\{x,y\} \in B$. In either case~$\zeta(\{x,y\})=-1$ or~$\varphi_\mathrm{edge}(\zeta,A,B) = 0$, using~\eqref{eta_zeta_eq}, we have that
    \begin{equation}\label{duality_proof_2}
        D_{\mathbf{x}, \mathbf{i}, A ,B}(\eta^{y \to x}_\zeta, \zeta) - D_{\mathbf{x}, \mathbf{i}, A ,B}(\eta, \zeta) =  D_{\eta,\zeta}( \mathbf{x}^{x \to y}, \mathbf i^{x,\mathbf x}, A,B ) - D_{\eta,\zeta}( \mathbf{x}, \mathbf{i}, A,B);        
    \end{equation}

\noindent \textbf{Unknown edge}, that is,~$\{x,y\} \in (A \cup B)^c$. First, assuming~$\zeta(\{x,y\}) = +1$, we have
    \begin{align*}
    D_{\mathbf{x}, \mathbf{i}, A ,B}(\eta^{y \to x}_\zeta, \zeta) & \stackrel{\eqref{psi_varphi_def},\eqref{eta_zeta_eq}} = p \cdot \varphi_\mathrm{edge}(\zeta, A \cup\{x,y\}, B) \cdot \varphi_\mathrm{site}(\eta,\mathbf x^{x \to y},\mathbf i)  \\[2mm]
    & \hspace{5mm} = p \cdot D_{\eta,\zeta}(\mathbf x^{x \to y},\mathbf i, A \cup \{x,y\}, B); 
    \end{align*}
    otherwise, if~$\zeta(\{x,y\})=-1$, analogously, we have
    $$D_{\mathbf{x}, \mathbf{i}, A ,B}(\eta^{y \to x}_\zeta, \zeta) = (1-p) \cdot D_{\eta,\zeta}(\mathbf{x}^{x \to y}, \mathbf i^{x,\mathbf{x}}, A, B \cup \{x,y\}). $$    
    Hence,
    \begin{equation}\label{duality_proof_3}
        \begin{aligned}
            D_{\mathbf{x}, \mathbf{i}, A ,B}(\eta^{y \to x}_\zeta, \zeta)& -  D_{\mathbf{x}, \mathbf{i}, A ,B}(\eta, \zeta) \\[2mm]
            & = p \cdot \big[ D_{\eta,\zeta}(\mathbf x^{x \to y},\mathbf i, A \cup \{x,y\}, B) - D_{\eta,\zeta}(\mathbf x,\mathbf i, A, B)  \big] \\[0.2cm]
            & \hspace{3mm} + (1-p) \cdot \big[ D_{\eta,\zeta}(\mathbf x^{x \to y},\mathbf i^{x,\mathbf x}, A, B\cup \{x,y\}) - D_{\eta,\zeta}(\mathbf x,\mathbf i, A, B)  \big].
        \end{aligned}
    \end{equation}

Together~\eqref{duality_proof_1},~\eqref{duality_proof_2} and~\eqref{duality_proof_3}, multiplying by~$Q(x,y)$ and summing over~$(x,y) \in V^2$ with~$x \sim y$ we obtain that
    $$ \sum_{(x,y) \in V^2: x \sim y } Q(x,y) \cdot \big[ D_{\mathbf{x},\mathbf{i},A,B}(\eta^{y \to x}_\zeta,\zeta) - D_{\mathbf{x},\mathbf{i},A,B}(\eta,\zeta)\big] $$
    equals 
    \begin{equation}\label{duality_proof_sum_site}
        \begin{aligned} 
            & \sum_{(x,y) \in V^2 : x \sim y } \big( Q(x,y) \cdot \big[ D_{\eta,\zeta}(\mathbf{x}^{x \to y},\mathbf{i}, A,B) - D_{\eta,\zeta}(\mathbf{x},\mathbf{i},A,B) \big] \\[0.2cm]
            & \hspace{2.5cm} + Q(x,y) \cdot \big[ D_{\eta,\zeta}(\mathbf{x}^{x \to y},\mathbf{i}^{x,\mathbf{x}}, A,B) - D_{\eta,\zeta}(\mathbf{x},\mathbf{i},A,B) \big] \\[0.2cm]
            & \hspace{2.5cm} + p \cdot Q(x,y) \cdot \big[ D_{\eta,\zeta}(\mathbf{x}^{x \to y},\mathbf{i}, A\cup\{x,y\},B) - D_{\eta,\zeta}(\mathbf{x},\mathbf{i},A,B) \big] \\[0.2cm]
            & \hspace{2.5cm} + (1-p) \cdot Q(x,y) \cdot \big[ D_{\eta,\zeta}(\mathbf{x}^{x \to y},\mathbf{i}^{x,\mathbf x}, A,B\cup\{x,y\}) - D_{\eta,\zeta}(\mathbf{x},\mathbf{i},A,B) \big] \big).
        \end{aligned}
    \end{equation}
    On the other hand, it is clear that
    \begin{equation}\label{duality_proof_sum_edge}
        \begin{aligned}
        \sum_{e \in E}  c(e,\zeta,\mathsf{v},p)  \cdot \big[  D_{\mathbf{x},\mathbf{i},A,B} &(\eta,\zeta^e) - D_{\mathbf{x},\mathbf{i},A,B}(\eta,\zeta) \big] \\
        & = \sum_{e \in E} \mathsf{v} \cdot \big[ D_{\eta,\zeta}( \mathbf{x},\mathbf{i}, A \setminus \{e\}, B \setminus \{e\}) - D_{\eta,\zeta}(\mathbf{x},\mathbf{i}, A, B) \big].
        \end{aligned}
    \end{equation}
    Therefore, the duality relationship~\eqref{duality_to_proof} follows by summing~\eqref{duality_proof_sum_site} to both sides of~\eqref{duality_proof_sum_edge}.
\end{proof}
An immediate consequence of the duality relationship is the following
\begin{corollary}
Fix~$k \in \N$. Then, for any configuration~$(\eta,\zeta,\mathbf x, \mathbf i,A,B ) \in \Omega \times \mathcal S_k$, with~$\mathbf x = (x_1,\dots,x_k)$ and~$\mathbf i = (\mathsf i_1,\dots, \mathsf i_k)$, and any time~$t \ge 0$,
\begin{equation}\label{duality_equality}
    \begin{aligned}
    P_{\eta,\zeta}( \eta_t(x_1)  =\mathsf i_1,\dots \eta_t(x_k) &=  \mathsf i_k, \zeta_t \equiv 1 \text{ on } A, \zeta_t \equiv -1 \text{ on } B ) \\[0.2cm]
    & = p^{|A|}(1-p)^{|B|} \cdot E_{\mathbf{x},\mathbf{i},A,B}[D(\eta,\zeta,\mathbf Y_t,\mathbf i_t,A_t,B_t) ].        
    \end{aligned}
\end{equation}
\end{corollary}
This duality relationship shows, as in the classical voter model, that the dynamics of the voters can be understood in terms of the ancestry of opinions. Tracing the adoption of an opinion backwards in time, the location of the ancestral opinion evolves as a random walk with jump rate~$Q$ on the underlying graph. Each time the backward trajectory crosses a negative edge, a sign flip is recorded, accounting for the effect of anti-voter interactions. The term involving~$p$ arises from the reversibility of the edge environment with respect to the measure~$\pi^\mathrm{edge}_p$.

When the ancestry of several sites is traced simultaneously, one obtains a system of random walks evolving backwards in time that coalesce upon meeting, reflecting the fact that once two sites share a common ancestor, their subsequent histories coincide. Along each backward path, the accumulated product of edge signs determines whether the opinion is transmitted or reversed. As a consequence, the joint distribution of opinions at a fixed time can be expressed in terms of the meeting structure of these coalescing random walks and the parity of sign changes accumulated along their trajectories. 

\begin{remark}\label{rmk_coalescing_synchronization}
Let~$(Y^1_s)$ and~$(Y^2_s)$ be two different random walks in~$(\mathbf Y_s)$ that coalesce at time~$t$. We say that their associated sign processes \emph{synchronize} if~$\mathsf i^1_t = \mathsf i^2_t$. Since the random walks have coalesced, they traverse the same edges from~$t$ onward. Consequently, the sign processes evolve identically from that time on, and therefore
$$\mathsf i^1_s = \mathsf i^2_s \quad \text{ for all } s \ge t.$$ 
By the same reasoning, if synhcronization does not occur, then~$\mathsf i^1_s \neq \mathsf i^2_s$ for all~$s \ge t$. In particular, in the absence of synchronization, the event~$\{ \eta( Y^1_s) = \mathsf i^1_s, \eta(Y^2_s) = \mathsf i^2_s \}$ is empty for all~$s \ge t$.
\end{remark}

\subsection{Stationary measures}
In this section, we use the duality relationship developed in the previous section to obtain a stationary measure of the voter model with dynamic anti-voter bonds on a graph~$G$ with adoption rate~$Q$ and environment parameters~$p \in (0,1)$ and~$\mathsf v$. In Section~\ref{ss_ergodicity} we prove that this is the unique stationary measure of the process.  

\begin{remark}\label{rmk_stationary_static_degenerate}
For the degenerate cases~$p \in \{0,1\}$, it is straightforward to verify that, for any underlying graph and any adoption rate, a measure~$\mu$ is stationary for the voter model on the static all-positive or all-negative signed environment if and only if the product measure~$\mu \otimes \pi^\mathrm{edge}_p$ is stationary for the voter model with dynamic anti-voter bonds.
\end{remark}

\noindent \textbf{A stationary measure.} Let~$\nu$ be the Bernoulli product measure on~$\Omega_\mathrm{site}$ of density parameter~$1/2$. Given~$k \in \N$, we have that for any configuration~$(\mathbf x, \mathbf i, A,B) \in \mathcal{S}_k$ with~$\mathbf i = (1,\dots,1)$ and any time~$t \ge 0$,
\begin{equation}\label{eq_a_stationary_measure_0}
    \begin{aligned}
    \int P_{\eta,\zeta}( \eta_t(x_1 &) =  1, \dots, \eta_t(x_k) = 1,  \zeta_t \equiv 1 \text{ on } A, \zeta_t \equiv -1 \text{ on } B ) \; \nu \otimes \pi^\mathrm{edge}_p ( \mathrm{d}(\eta,\zeta) ) \\[0.2cm]
    & = p^{|A|}(1-p)^{|B|} \cdot \int E_{\mathbf{x},\mathbf{i},A,B}[D(\eta,\zeta,\mathbf Y_t,\mathbf i_t,A_t,B_t)] \; \nu \otimes \pi^\mathrm{edge}_p ( \mathrm{d}(\eta,\zeta) ) \\[0.2cm]
    & = p^{|A|}(1-p)^{|B|} \cdot E_{\mathbf{x},\mathbf{i},A,B} \big[\nu \big(\eta: \eta(Y^1_t) = \mathsf i^1_t,\dots,\eta(Y^k_t)= \mathsf i^k_t \big) \big],
    \end{aligned}
\end{equation}
where the first equality follows from the duality relationship~\eqref{duality_equality} and the second one by applying Fubini's theorem.

We define the following set of indices and events:
\begin{align*}
& \mathcal C_n := \{ m \in \{1,\dots,k\} : \exists t \ge 0 \text{ such that } Y^m_t =Y^n_t \};  \\[0.2cm]
& I_n := \{ \forall m \in \mathcal C_n \; \exists \; T \ge 0 \text{ such that } \mathsf i^m_t = \mathsf i^n_t \text{ for all } t \ge T  \}, \quad n = 1,\dots,k.
\end{align*}
In words, the set~$\mathcal C_n$ consists of the indices of random walks that coalesce with~$(Y^n_t)_{t \ge 0}$, while~$I_n$ is the event that their corresponding sign processes synchronize with that of~$(Y^n_t)_{t \ge 0}$. See Remark~\ref{rmk_coalescing_synchronization}. 

Note that~$\{\mathcal C_1,\dots,\mathcal C_k\}$ forms a partition of the set~$\{1,\dots,k\}$, determined by the coalescence structure of the random walks~$(Y^1_t,\dots,Y^k_t)_{t \ge 0}$. Also observe that the event~$\bigcap_{n=1}^k I_n$ corresponds to the case where, within each coalesced group, the associated sign processes are synchronized.

From this observation, it is easy to see that
\begin{equation}\label{eq_a_stationary_mesure}
    \lim_{t \to \infty} E_{\mathbf{x},\mathbf{i},A,B} \big[\nu \big(\eta(Y^1_t) = \mathsf i^1_t,\dots,\eta(Y^k_t)= \mathsf i^k_t \big) \big] = E_{\mathbf{x},\mathbf{i},A,B} \Big[ \Big(\frac{1}{2}\Big)^{| \{\mathcal C_1,\dots, \mathcal C_k\} |} \cdot \prod^k_{n=1} \mathds{1}_{I_n} \Big].
\end{equation}
Therefore, taking the limit as~$t \to \infty$ in~\eqref{eq_a_stationary_measure_0} and applying~\eqref{eq_a_stationary_mesure}, we conclude that under~$\nu \otimes \pi^\mathrm{edge}_p$, the process~$(\eta_t,\zeta_t)$ converges in distribution to a probability measure~$\mu^\mathrm{dyn}$ on~$\Omega$ characterized by
\begin{equation}\label{mu_dyn_def}
    \begin{aligned}
        \mu^\mathrm{dyn} \big( (\eta,\zeta): \eta(x_1)= 1, \dots & , \eta(x_k)= 1 ,  \zeta \equiv 1 \text{ on } A ,  \zeta \equiv -1 \text{ on } B\big) \\
        & = p^{|A|}(1-p)^{|B|} \cdot E_{\mathbf{x},\mathbf{i},A,B} \Big[ \Big(\frac{1}{2}\Big)^{| \{\mathcal C_1,\dots, \mathcal C_k\} |} \cdot \prod^k_{n=1} \mathds{1}_{I_n} \Big]
    \end{aligned}
\end{equation}
for any~$(\mathbf x, \mathbf i, A,B) \in \mathcal S_k$ with~$\mathbf i = (1,\dots,1)$. As the measure~$\mu^\mathrm{dyn}$ is obtained as a distributional limit of~$(\eta_t,\zeta_t)$, it is stationary for this process. This follows from the Feller property of the process, as stated in~\cite[Proposition I.1.8.(d)]{Liggett2005}.

\section{Ergodicity}\label{ss_ergodicity}
As mentioned in the Introduction, we follow Liggett’s notion of ergodicity. We now introduce some notation and restate this definition in a formal way. Let~$(\xi_t)_{t \ge 0}$ be a Markov process with state space~$\mathcal C$, a compact metric space. For a probability measure~$\mu$ on~$\mathcal C$ and a time~$t \ge 0$, we denote by~$\mu S_t$ the distribution of~$\xi_t$ under~$P_\mu$. Then~$(\xi_t)_t$ is ergodic if:
\begin{enumerate}[label=(\Alph*)]
    \item it has a unique stationary probability measure~$\mu$; and
    \item for any probability measure~$\nu$ in~$\mathcal C$,~$\nu S_t$ converges weakly to~$\mu$ as~$t \to \infty$. 
\end{enumerate}

We say that a Markov process has the Feller property, and refer to it as a~\emph{Feller process}, if
$$ \xi \mapsto \E_\xi[ f(\xi_t) ] \quad \text{is continuous for every continuous } f:\mathcal C \to \R \text{ and } t \ge 0.$$
It is well known that, for Feller processes, condition (B) implies condition (A). However, the converse implication does not hold in general. As a counterexample, one may consider the example given in~\cite{Liggett2005}: a deterministic process that moves at uniform speed around the unit circle. In this case, the normalized Lebesgue measure on the circle is the unique stationary probability measure, yet there exists a sequence~$t_n \to \infty$ such that for any probability measure~$\mu$, we have~$\mu S_{t_n}=\mu$ for all~$n$. Hence, convergence to the Lebesgue measure fails.

In contrast, for certain Feller processes, including the one studied here, condition~(A) does imply condition~(B). We establish this implication in Section~\ref{ss_uniq_implies_ergo}. Consequently, to prove that the voter model with dynamic anti-voter bonds is ergodic, it suffices to show that the measure~$\mu^\mathrm{dyn}$, defined in~\eqref{mu_dyn_def}, is the unique stationary measure of the process. This is proved in Section~\ref{ss_uniqueness} by means of a coupling between the dual coalescing system and an auxiliary process consisting of independent random walks combined with dynamic signs and edge explorations, introduced in Section~\ref{ss_ind_rw_coupling}.

Throughout this section, we fix a graph~$G = (V,E)$ and a function~$Q$ satisfying~\eqref{Q_def_adoption_rw}, as well as the environment parameters~$p \in (0,1)$ and~$\mathsf v \in (0,\infty)$. As noted in Remark~\ref{rmk_stationary_static_degenerate}, the study of ergodicity in the degenerate cases~$p \in \{0,1\}$ reduces to the corresponding static environments. A review of these results can be found in the Introduction. 

\subsection{Uniqueness implies ergodicity}\label{ss_uniq_implies_ergo}
We follow the approach of Theorem~1.4 in~\cite{MaillardMountford2013} to show that the uniqueness of the stationary measure implies ergodicity. We extend this approach to more general Feller processes defined on metric spaces with a dual, which is a Markov chain with a countable state space. Define
\[
\mathscr{I}_\infty := \left\{\begin{array}{l}\mu:  \text{there exists a probability measure $\nu$ on $\mathcal C$ and a sequence $(t_n)_{n \ge 0}$}\\
\text{with $t_n \to \infty$ such that $\nu S_{t_n}$ converges weakly to $\mu$ as $n \to \infty$}\end{array}\right\}.
\]
\begin{lemma}\label{lemma_uniqueness_1}
Let~$(\xi_t)_{t \ge 0}$ be a Feller process defined on a metric space~$\mathcal C$ and~$(Z_t)_{t \ge 0}$ be a Markov chain with a countable state space~$\mathcal S$ and~$Q
$-matrix~$\{q(x,y) : x,y \in \mathcal S\}$ that satisfy
\begin{equation}\label{eq_q_finite}
c(x) > 0, \; \forall x \in \mathcal S, \quad
\sup_{ \substack{ x,y \in \mathcal S \\ q(x,y) > 0 } } \frac{c(x)}{c(y)} < \infty, \quad \int^\infty_0 \mathds{1}_{ \{Z_{t^-} \neq Z_t\} } \cdot c(Z_t)^{-2} \; \mathrm{d}t = \infty \; a.s.,
\end{equation}
where~$c(x) = -q(x,x)$. Assume that the processes are dual with respect to a function~$D:\mathcal C \times \mathcal S \to [0,\infty)$, that is,
$$ E_\xi D(\xi_t,z) = E_z D(\xi_t,z) \quad \text{ for every } \xi \in \mathcal C,\; z \in \mathcal S,\; t \ge 0; $$
also assume that for each~$z \in \mathcal S$,
\begin{enumerate}[label=(\roman*)]
    \item the map~$\xi \mapsto D(\xi,z)$ is bounded and continuous; and
    \item there exit~$M>0$ and~$\alpha>0$ such that for each~$\xi \in \mathcal C$ and~$t \ge 0$ we have that $E_z [D(\xi,Z_t)^{1+\alpha}] \leq M$.
\end{enumerate} 
Moreover, assume that 
$\mu \in \mathscr{I}_\infty$. Then,
\begin{equation}\label{eq_v_stat_unique}
    \int D(\xi,z)\; \mu( \mathrm{d} \xi ) = \int D(\xi,z) \; \mu S_t( \mathrm{d} \xi ) \quad \text{for any } t \ge 0 \text{ and } z \in \mathcal S.
\end{equation}
\end{lemma}  
The proof of this lemma is postponed to the end of this section. We now state the following immediate consequence.
\begin{corollary}\label{cor_uniqueness_implies_ergodicity}
Let~$(\xi_t)_{t \ge 0}$ be a Feller process on a metric space~$\mathcal C$. Let~$I$ be an index set and, for each~$i \in I$, let~$(Z^i_t)_{t \ge 0}$ be a Markov chain with countable state space~$\mathcal S_i$ and~$Q$-matrix satisfying~$\eqref{eq_q_finite}$. Suppose that, for every~$i \in I$, the processes~$(\xi_t)$ and~$(Z^i_t)$ are dual with respect to~$D_i:\mathcal C \times \mathcal S_i \to [0,\infty)$, which satisfies conditions~$(i)$ and~$(ii)$ of Lemma~\ref{lemma_uniqueness_1} for all~$z \in \mathcal S_i$.
\begin{enumerate}
    \item[(1)] If, in addition, the family of functions~$\xi \mapsto D_i(\xi,z)$, indexed by~$i \in I$ and~$z \in \mathcal S_i$, is measure-determining (that is, two probability measures coincide whenever their integrals against all these functions coincide), then every measure in~$\mathscr{I}_\infty$ is stationary for~$(\xi_t)_t$.
    \item[(2)] If, furthermore, the metric space~$\mathcal C$ is compact and~$(\xi_t)_t$ admits a unique stationary measure, then $(\xi_t)_t$ is ergodic.
\end{enumerate}
\end{corollary}
\begin{remark}
By Proposition~I.1.8(d) in~\cite{Liggett2005}, every measure in
\[
\left\{\begin{array}{l}\mu:  \text{there exists a probability measure $\nu$ on $\mathcal C$}\\
\text{such that $\nu S_{t}$ converges weakly to $\mu$ as $t \to \infty$}\end{array}\right\}.
\]
is stationary for~$(\xi_t)_t$. Note that the above set is contained in~$\mathscr{I}_\infty$.
\end{remark}
\begin{proof}[Proof of Corollary~\ref{cor_uniqueness_implies_ergodicity}] Part~(1) is immediate from the measure-determining property applied to~\eqref{eq_v_stat_unique}, which implies that, for any~$\mu \in \mathscr{I}_\infty$, we have~$\mu = \mu S_t$ for all~$t \ge 0$. Then,~$\mu$ is stationary. For part~(2), let~$\mu$ denote the unique stationary measure. By part~(1) we then have~$\mathscr{I}_\infty = \{\mu\}$. Fix a measure~$\nu$ in~$\mathcal C$ and an increasing sequence~$(t_n)_{n \ge 0}$ with~$t_n \to \infty$. It remains to show that the measures~$\nu S_{t_n}$ converge weakly to~$\mu$ as~$n \to \infty$. 

Recall a standard fact: a sequence of probability measures on a metric space converges to a probability measure~$\kappa$ if every subsequence has a further subsequence that converges to~$\kappa$. To apply this criterion, we use the compactness of~$\mathcal C$ and apply Prokhorov's theorem to obtain a subsequential limit, which belongs to~$\mathscr{I}_\infty$. Hence, the sequence converges to~$\mu$, and the proof is complete.
\end{proof}
We now apply this corollary to our model to establish a key result needed for the proof of the main theorem.
\begin{lemma}\label{lemma_unique_implies_ergo}
Consider the voter model with dynamical anti-voter bonds defined on an arbitrary graph~$G$, with adoption rates given by a~$Q$-matrix satisfying~\eqref{Q_def_adoption_rw}, and with environment parameters~$p \in (0,1)$ and~$\mathsf v \in (0,\infty)$. If this process has a unique stationary measure, then it is ergodic.
\end{lemma}
\begin{proof}
Recall the processes and the duality function~$D_k$ appearing in Lemma~\ref{lemma_duality_raw}. It is clear that the family
$$\{D_k(\; \cdot \;,\mathbf x, \mathbf{i}, A, B ): k \in \N, (\mathbf x, \mathbf i, A,B) \in \mathcal S_k \}$$
is measure-determining. Consequently, by Corollary~\ref{cor_uniqueness_implies_ergodicity}, it suffices to verify that the duality function~$D_k$ satisfies the properties stated in Lemma~\ref{lemma_uniqueness_1}. 

Fix~$k \in \N$. We start by verifying that the dual process~$(\mathbf Y_t,\mathbf i_t,A_t,B_t)_{t \ge 0}$ from~\eqref{matrix_q_crw_dyn} satisfies~\eqref{eq_q_finite}. Note that
$$ c( \mathbf x, \mathbf i, A,B ) = \mathsf v \cdot |A \cup B| + | \{ x_1,\dots,x_k \} | \ge 1, \quad (\mathbf x, \mathbf i, A, B) \in \mathcal S_k.$$
On the other hand, we have that
$$c( \mathbf x, \mathbf i, A,B ) \leq \max\{2,1+\mathsf v \} \cdot (\mathbf x', \mathbf i', A',B') \quad \forall \; Q_\mathrm{crw}\big( (\mathbf x, \mathbf i, A,B),(\mathbf x', \mathbf i', A',B') \big) > 0. $$
To verify the last property, we focus on the evolution of the positive and negative edges revealed by the random walk, we observe that~$\big(|A_t \cup B_t|\big)_{t \ge 0}$ can be stochastically dominated by a birth–death process~$(R_t)_{t \ge 0}$ on~$\mathbb{N}_0$ with transition rates
$$ k \to k+1 \quad \text{at rate } 1, \quad k \to k-1 \quad \text{at rate } \mathsf v k \quad \text{for every } k \in \N.$$
From this, it follows that the set~$\mathcal R = \{t: A_{t^-} \cup B_{t^-} \neq A_t \cup B_t = \varnothing\}$ is infinite almost surely, and then
$$ \int^\infty_0 \mathds{1}_{ \{ Z_{t^-} \neq Z_t \} } \cdot c(Z_t)^{-2} \; \mathrm{d}t \ge \sum_{t \in \mathcal{R}} c(Z_t)^{-2} \ge k \cdot |\mathcal R|.  $$

We now verify that~$D_k$ satisfies Property~$(i)$ and~$(ii)$. Property~$(i)$ is immediate; we focus on property~$(ii)$. For any~$(\eta,\zeta) \in \Omega$,~$(\mathbf x,\mathbf i, A,B) \in \mathcal S_k$ and~$t \ge 0$, 
$$ E_{\mathbf x, \mathbf i, A,B}[D_k( \eta,\zeta,\mathbf Y_t,\mathbf i_t, A_t,B_t )^2] \leq E_{\mathbf x,\mathbf i,A,B}[ p^{-2|A_t|}(1-p)^{-2|B_t|} ] \leq E_{\mathbf x, \mathbf i, A, B}[ e^{c_p \cdot |A_t \cup B_t|} ].$$
where~$c_p = -2\ln(\min\{p,1-p\})>0$. We now use the stochastic comparison with~$(R_t)$. Note that for every~$\theta \in \R$ and initial condition~$R_0 \in \N$,
$$ \E[ \exp(\theta \cdot |R_t| ) ] = (1- e^{-\mathsf v t} + e^{-\mathsf v t} e^\theta)^{R_0} \cdot \exp \left\{ \frac{1-e^{-\mathsf v t}}{\mathsf v}(e^\theta -1 ) \right\} \xrightarrow[]{t \to \infty} \exp \{(e^\theta-1)/\mathsf v \}. $$
Therefore, taking~$\theta = c_p$ and~$R_0 = |A \cup B|$, we conclude that there exists a constant~$M = M(p,\mathsf v,A,B)$ such that~$(ii)$ holds with~$\alpha = 1$. This completes the proof.
\end{proof}

We now prove Lemma~\ref{lemma_uniqueness_1}. To this end, we recall the following ``time-shift'' coupling from~\cite[Proposition~1.1]{Mountford93}. Observe that~\eqref{eq_q_finite} implies conditions~A, B, and~C of the article and then the conclusion follows.

\begin{lemma}\label{lemma_couopling_X+T}
Consider a Markov chain with a countable state space~$\mathcal{S}$ and~$Q$-matrix~$\{q(x,y): x,y \in \mathcal S\}$ that satisfy~\eqref{eq_q_finite}. Given~$t \in (0,\infty)$ and~$\epsilon>0$, there exists a finite~$s_0$ so that for any~$z \in \mathcal S$, there exists a coupling on~$\Prob$ of two Markov chains~$(Z^1_s)_{s \ge 0}$ and~$(Z^2_s)_{s \ge 0}$ with~$Q$-matrix~$\{q(x,y): x,y \in \mathcal S\}$ started both from~$z$ such that
$$ \Prob( Z^1_s = Z^2_{s+t} \text{ for all } s \ge s_0 ) > 1-\epsilon. $$
\end{lemma}
\begin{proof}[Proof of Lemma~\ref{lemma_uniqueness_1}]
Fix~$z \in \mathcal S$ and a probability measure~$\nu$ on~$\mathcal C$ such that~$\nu S_t$ converges weakly to~$\mu$ as~$t \to \infty$. By the duality relationship, we have that
\begin{equation}\label{eq_ergo_uniqueness_0}
E_\xi D(\xi_{t_n},z) =  E_z D(\xi,Z_{t_n}) \quad \text{for any } \xi \in \mathcal C \text{ and } n \in \N.    \end{equation}
We integrate both sides with respect to~$\nu$ and take the limit as~$n \to \infty$. On the left-hand side, assumption $(i)$ and the Feller property imply that
$$\int D(\xi,z) \; \mu( \mathrm{d}\xi ) = \lim_{n \to \infty} \int E_\xi D(\xi_{t_n},z) \; \nu( \mathrm{d}\xi ) \stackrel{\eqref{eq_ergo_uniqueness_0}} = \lim_{n \to \infty} \int E_z D(\xi,Z_{t_n}) \; \nu(\mathrm{d} \xi).$$
Note that assumption~$(ii)$ implies that~$\int D(\xi,z) \; \mu(d\xi) < \infty$. Now, we want to prove that for any~$t \ge 0$,
\begin{equation}\label{eq_final_meta_theorem}
    \lim_{n \to \infty} \int E_z D(\xi,Z_{t_n + t}) \; \nu(d\xi) = \int D(\xi,z) \; \mu(d \xi).
\end{equation}
This is sufficient since, by first applying the Markov property and then the duality relation, we obtain
$$ \int D(\xi,z) \; \mu S_t(\mathrm{d}\xi) = \lim_{n \to \infty} \int D(\xi,z) \; \nu S_{t_n+t}(\mathrm{d}\xi) = \lim_{n \to \infty} \int E_zD(\xi,Z_{t_n+t}) \; \nu(d\xi).$$
Substituting~\eqref{eq_final_meta_theorem} into the last expression completes the proof. 

We now turn to the proof of~\eqref{eq_final_meta_theorem}. To do so, we apply Lemma~\ref{lemma_couopling_X+T}. Given~$\epsilon > 0$, there exists a time-shift coupling event~$A$ with probability larger than~$1-\epsilon$ and~$N \in \N$ such that for any~$n \ge N$ and~$\xi \in \mathcal C$,
$$ \big| E_z[D(\xi,Z_{t_n + t})] - E_z[D(\xi,Z_{t_n})] \big| = \E[ |D(\xi,Z_{t_n+t}) - D(\xi,Z_{t_n})| \cdot \mathds{1}_{A^c} ] \leq (2M\epsilon^\alpha)^\frac{1}{1+\alpha}, $$
where the last inequality follows from Hölder's inequality and assumption~$(ii)$. A final application of the triangle inequality gives~\eqref{eq_final_meta_theorem}.
\end{proof}

\subsection{Independent random walks}\label{ss_ind_rw_coupling}
In view of Lemma~\ref{lemma_unique_implies_ergo}, to establish ergodicity, it remains to show that~$\mu^\mathrm{dyn}$ is the unique stationary measure of the process. To this end, we construct in Section~\ref{ss_uniqueness} a coupling between the Markov chain~$Q_{\mathrm{crw}}$ in~\eqref{matrix_q_crw_dyn} and a related system of independent random walks. In this section, we introduce this auxiliary process and establish a couple of preliminary results that will be used in the proof of uniqueness.  \\

\noindent \textbf{The sign process and independent random walks on dynamical signed edges.} Fix~$k \in \N$. Given~$\mathbf x \in V^k$ and~$\mathbf i \in \{-1,+1\}^k$, we write~$\mathbf x^{i,y} \in V^k$ and~$\mathbf i^i \in \{-1,+1\}^k$ for the vectors obtained from~$\mathbf{x}$ and~$\mathbf{i}$, respectively, by replacing the~$i$-th coordinate of~$\mathbf{x}$ with~$y \in V$, and flipping the sign of the~$i$-th coordinate of~$\mathbf i$, that is, multiplying it by~$-1$.

We define the Markov chain~$(X^1_t,\dots,X^k_t,\mathsf i^1_t,\dots, \mathsf i^k_t,A_t, B_t)_{t \ge 0}$ with state space~$\mathcal S_k$ and matrix rate~$Q_\mathrm{irw} = Q_\mathrm{irw}(p,\mathsf v,k,Q)$, which is defined as follows. Let~$(\mathbf x,\mathbf i, A,B) \in \mathcal S_k$ and denote by~$x_i$ the~$i$-th coordinate of~$\mathbf x$. 
\begin{equation}\label{matrix_q_irw_dyn}
    \begin{array}{ll}
        Q_\mathrm{irw} \big( ( \mathbf x,\mathbf i,A,B), ( \mathbf x^{i,y},\mathbf i,A,B ) \big) = Q(x_i,y), & \{x_i,y\} \in A, \; y \in V;  \\[0.2cm]
        Q_\mathrm{irw} \big( ( \mathbf x,\mathbf i,A,B), ( \mathbf x^{i,y},\mathbf i,A \cup \{x_i,y\},B ) \big) = p \cdot Q(x_i,y), & \{x_i,y\} \in (A \cup B)^c, \; y \in V; \\[0.2cm] 
        Q_\mathrm{irw} \big( ( \mathbf x,\mathbf i,A,B), ( \mathbf x^{i,y},\mathbf i^i,A,B ) \big) = Q(x_i,y), & \{x_i,y\} \in B, \; y \in V;\\[0,2cm]
        Q_\mathrm{irw} \big( ( \mathbf x,\mathbf i,A,B), ( \mathbf x^{i, y},\mathbf i^i,A,B \cup \{x_i,y\} ) \big) = (1-p) \cdot Q(x_i,y), & \{x,y\} \in (A \cup B)^c, \; y \in V; \\[0.2cm]
        Q_\mathrm{irw} \big( ( \mathbf x,\mathbf i,A,B), ( \mathbf x,\mathbf i,A \setminus \{e\},B \setminus \{e\} ) \big) = \mathsf v, & e \in A \cup B.
    \end{array}
\end{equation}

We now state and prove the results. As we did earlier, we write~$\mathbf X_t := (X^1_t,\dots,X^k_t)$ and~$\mathbf{i}_t := (\mathsf i^1_t, \dots, \mathsf i^k_t)$ for any~$t \ge 0$.
\begin{lemma}\label{lemma_limit_1/2}
Fix~$k \in \N$. Then, for every~$\eta \in \Omega_\mathrm{site}$ and~$(\mathbf x, \mathbf i ,A,B) \in \mathcal S_k$ we have that
    \begin{equation}\label{eq_1_prof_ergo_limit_final}
        \lim_{t \to \infty} P_{\mathbf x,\mathbf i,A,B}( \eta(X^1_t) = \mathsf i^1_t, \dots, \eta(X^k_t) = \mathsf i^k_t \mid A_t,B_t, \;(\mathbf X_s)_{s \ge 0} ) = \frac{1}{2^k} \; \text{ a.s.}
    \end{equation}    
\end{lemma}
\begin{remark}
If we apply the bounded convergence theorem in~\eqref{eq_1_prof_ergo_limit_final}, we have that
$$ \lim_{t \to \infty} P_{\mathbf x,\mathbf i} \big( \eta(X^1_t) = \mathsf i^1_t, \dots, \eta(X^k_t) = \mathsf i^k_t \big) = \frac{1}{2^k} \quad \text{for any } \textbf{x} \in V^k \text{ and } \textbf{i} \in \{-1,+1\}^k $$
In the static scenario, independence of the processes~$(X^i_t,\mathsf i^i_t)$ implies that the last display is equivalent to
$$\lim_{t \to \infty}P_{x,\mathsf i}( \eta_t(x) = \mathsf i ) = 1/2, \quad \text{for any } x\in V \text{ and } \mathsf i \in \{-1,+1\},$$ 
which ensures ergodicity for the voter model with anti-voter bonds, as shown in Theorem~6.1 of~\cite{Gantertetall2005}. Obtaining an analogous result to the case of a random environment requires additional work, as the processes~$(X^i_t, \mathsf i^i_t)$ are no longer independent.
\end{remark}
\begin{proof}[Proof of Lemma \ref{lemma_limit_1/2}]

By the path conditioning~$(\textbf{X}_s)_{s \ge 0}$, we have that
$$ E_{\mathbf{x},\mathbf{i},A,B}[ \varphi_\mathrm{site}(\eta,\mathbf{X}_t,\mathbf{i}_t) \mid A_t,B_t,(\mathbf{X_s})_{s \ge 0} ] = \hspace{-3mm} \sum_{\mathbf j \in \{-1,+1\}^k } \hspace{-4mm}\varphi_\mathrm{site}( \eta,\mathbf{X}_t,\mathbf{j}) \cdot P_{\mathbf{x},\mathbf{i},A,B}( \mathbf i_t = \mathbf j \mid A_t,B_t,(\mathbf{X_s})_{s \ge 0} )  $$
So, to obtain~\eqref{eq_1_prof_ergo_limit_final} it suffices to prove that for any~$(\mathbf x, \mathbf i, A,B) \in \mathcal S_k$ and~$\mathbf j \in \{-1,+1\}^k$, 
$$ \lim_{t \to \infty} P_{\mathbf{x},\mathbf{i},A,B} ( \mathbf i_t = \mathbf j \mid A_t,B_t,(\mathbf{X_s})_{s \ge 0} ) = \frac{1}{2^k} \; \text{ a.s.}$$
First, let us focus on the case~$k = 1$. The extension to multiple entries follows similar arguments with minor modifications. We give some details at the end of the proof. Fix~$x \in V$ and~$\mathsf i \in \{-1,+1\}$, we aim to prove that
\begin{equation}\label{eq_1_prof_ergo_limit}
    \lim_{t \to \infty} P_{x,\mathsf i,A,B}( \mathsf i_t = +1 \mid A_t,B_t, \;(X_s)_{s \ge 0} ) = \frac{1}{2} \; \text{ a.s.}
\end{equation}    

We start by giving an alternative formulation for the sign process~$\mathsf i_t, \; t \ge 0$. Whenever the random walk crosses an unknown edge at time~$t$, that is~$\{X_{t^-},X_t\} \notin A_{t^-} \cup B_{t^-}$, we sample a random variable~$\xi$ independent of~$\sigma \big( X_s, \mathsf i_s, A_s, B_s : s \in [0,t) \big)$ with distribution
\begin{equation}\label{distribution_xi}
    p \cdot \delta_{\{+1\}} + (1-p)\cdot \delta_{\{-1\}}.
\end{equation}
This random variable is then used to update the value of the sign process and the environment process as follows:
    $$ \mathsf i_t = \mathsf i_{t^-} \cdot \xi  \quad \text{ and } \; \left\{ \begin{array}{ll}
        A_t = \{X_{t^-},X_t\} \cup A_{t^-} \text{ and } B_t = B_{t^-}  & \text{ if } \xi = 1; \\[0.2cm]
        B_t = \{X_{t^-},X_t\} \cup B_{t^-} \text{ and } A_t = A_{t^-}  & \text{ if } \xi = -1.
    \end{array} \right.$$
After time~$t$, the random walk may traverse the edge~$\{X_{t^-},X_t\}$ again before it is~\emph{refreshed}, that is, when the edge becomes unknown. In such cases, we do not resample a new random variable; instead, we reuse the same value of~$\xi$ to update the sign process. Since the time until the edge refreshes is exponentially distributed with rate~$\mathsf v$, the number of times it is crossed before refreshing is almost surely finite. This observation also implies that, for any initial environment~$(A_0,B_0)$, the random walk traverses infinitely many unknown edges along its trajectory.

We consider an increasing sequence of times~$\{t_i\}_{i \in \mathbb{N}}$ corresponding to the moments when the random walk traverses the unknown edges~$e_i$, at which point it draws the random variable~$\xi_i$ with distribution~\eqref{distribution_xi}. From the previous discussion, we have that for any~$t \ge 0$,    
$$\mathsf i_t = \mathsf i_0 \cdot \xi_0(t) \cdot \prod^{U_t}_{i =1} \xi_i^{n_i(t)},$$
where~$\xi_0(t) = \xi_0(A_0,B_0,t)$ denotes the update of~$\mathsf i_t$ given by the edges in~$A_0 \cup B_0$ that have been crossed by the random walk before their refreshing time in~$[0,t]$,~$U_t$ is the number of unknown edges traversed by the random walk up to time~$t$, and~$n_i(t)$ denotes the number of times the random walk traverses the edge~$e_i$ before refreshing in~$[t_i,t]$, including the traversal at time~$t_i$. To obtain~\eqref{eq_1_prof_ergo_limit}, we use the following facts: \\

\noindent \textbf{1. Asymptotic independence with the environment.} Given~$\ell \in \N$, for any function~$g:\{-1,+1\}^{\ell} \times \N^\ell \to \R$ we have that, conditioned on the entire trajectory of the random walk~$(X_s)_{s \ge 0}$,
    \begin{equation}\label{Eq_fact_1}
        \begin{aligned}
            \lim_{t \to \infty} E_{x,\mathsf i,A,B} \big[g  \big( & \xi_1 ,\dots,  \xi_\ell,n_1(t),\dots,n_\ell(t) \big)  \mid A_t,B_t \big] \\[2mm]
            & = E_{x,\mathsf i,A,B}\big[g \big( \xi_1,\dots,\xi_\ell,n_1(\infty),\dots,n_\ell(\infty)\big) \big], \quad \text{a.s.},
    \end{aligned}
    \end{equation}
    where~$n_i(\infty)$ denotes the total number of times the random walk traverses the edge~$e_i$ before refreshing after time~$t_i$, including the traversal at time~$t_i$. \\
    
\noindent \textbf{2. Infinitely many one-time crossings of unknown edges.}  Conditioning on the entire trajectory of the random walk~$(X_s)_{s \ge 0}$, we have that
    \begin{align*}
    P_{x,\mathsf i,A,B}( n_i(\infty) = 1 \mid (X_s)_{s \ge 0} ) & \ge P_{x,\mathsf i,A,B} \left(\left. \begin{array}{c}
    (X_s) \text{ traverses } e_i \text{ and this edge}  \\
    \text{refreshes before the next jump of } (X_s) 
    \end{array} \;\right|\; (X_s)_{s \ge 0} \right)  \\[2mm]
    & = 1 - \exp( - \mathsf v \cdot \tau_i ) \ge (1-\exp(-\mathsf{v})) \cdot \mathds{1}\{ \tau_i \ge 1 \},  
    \end{align*}
    where~$\tau_i$ is the holding time of the random walk after its crossing jump at time~$t_i$. Note that~$\{\tau_i: i \in \N\}$ is an i.i.d. sequence of exponential random variables of parameter~$1$. Then, by the second lemma of Borel-Cantelli, we can see that
    $$ \sum^\infty_{i =1} P_{x,\mathsf i,A,B}( n_i(\infty) = 1 \mid (X_s)_{s \ge 0}) \ge (1-\exp{(-\mathsf v)}) \cdot \sum^\infty_{i = 1} \mathds{1}\{ \tau_i \ge 1 \} = \infty \quad \text{a.s.}$$
    We now apply the conditional version of the Borel–Cantelli lemma, as stated in Lemma~3.3 of~\cite{Majerek05}. Since the random variables~$\{ n_i(\infty): i \in \N \}$ are independent conditionally on~$\sigma(X_s:s\ge 0)$, the lemma implies that
\begin{equation}\label{Eq_fact_2}
    P_{x,\mathsf i,A,B} \left(\left.\sum^\infty_{i = 1} \mathds{1}\{ n_i(\infty) = 1 \} = \infty \;\right| \;(X_s)_{s \ge 0} \right) = 1 \quad \text{a.s.}
\end{equation}

\noindent \textbf{3. Standard fact.} Let~$\{Z_i: i \in \N \}$ be an i.i.d. sequence of random variables with distribution~\eqref{distribution_xi}. We have that
\begin{equation}\label{Eq_fact_3}
    \lim_{\ell \to \infty} \Prob \left(\prod^\ell_{i=1}Z_i = 1 \right) = \frac{1}{2}.
\end{equation}

We now use these facts to complete the proof. Given~$\epsilon > 0$ we want to find~$T>0$ such that
\begin{equation}\label{eq_1_1_prof_ergo_limit}
    \left| P_{x,\mathsf i,A,B}( \mathsf i_t = 1 \mid A_t,B_t,(X_s)_{s \ge 0} ) - \frac{1}{2} \right| < \epsilon \quad \text{for all } t > T.
\end{equation}
Recall the increasing sequence~$t_i$ of revealing times where the random variables~$\xi_i$ were sampled. Let~$m_\ell$ be the smallest integer number such that~$\{\xi_i: i=1,\dots,m_\ell\}$ has~$\ell$ variables~$\xi_i$ satisfying~$n_i(\infty) = 1$. By~\eqref{Eq_fact_2}, for any~$\ell$, one can choose~$m_\ell$ with this property. We now apply the standard fact~\eqref{Eq_fact_3} to obtain~$\ell$ such that:   
$$ \big| P_{x,\mathsf i,A,B} \left( \Xi_\ell = 1 \mid (X_s)_{s \ge 0}  \right) - 1/2 \big| < \epsilon/2, \quad \text{where} \quad \Xi_\ell := \sum^{m_\ell}_{i=1,\; n_i(\infty)=1} \xi_i. $$
Moreover, it follows from~\eqref{Eq_fact_1} that there exists~$T_1>0$ such that
$$\big| P_{x,\mathsf i,A,B} \left( \Xi_\ell = 1 \mid A_t,B_t,(X_s)_{s \ge 0}  \right) - 1/2 \big| < \epsilon \quad \text{for all } t > T_1.$$
On the other hand, for any~$t > t_{m_\ell}$ we can write~$\mathsf i_t = \Xi_\ell \cdot \mathsf i^\ast_t$, where
$$ \mathsf i^\ast_t = \mathsf i_0 \cdot \xi_0(t)  \cdot \prod^{U_t}_{i = 1,\; n_i(\infty) \ge 2} \xi^{n_i(t)}_i $$
Note that~$\Xi_\ell$ and~$\mathsf i^\ast_t$ are conditionally independent for any~$t$. Then, for any~$t > \max\{ t_{m_\ell},T_1 \}$ we have~\eqref{eq_1_1_prof_ergo_limit} and thus we have~\eqref{eq_1_prof_ergo_limit}.

When considering several independent sign processes~$(\mathsf i^1_t),(\mathsf i^2_t),\dots,$ evolving on the same environment~$(A_t,B_t)$ only minor extensions are required. Each sign process can be constructed using the same formulation with independent families of random variables, each having distribution~\eqref{distribution_xi}. Since an extension of~\eqref{Eq_fact_1} remains valid for multiple coordinates corresponding to different processes, the same arguments applies and yield the result for several entries.
\end{proof}
As a consequence, we have the following result. Recall the definition of the duality function~$D$ from Lemma~\ref{lemma_duality_raw}.
\begin{corollary}\label{cor_ergo_limit}
    Fix~$k \in \N$. Then, for every~$(\mathbf{x},\mathsf i ,A,B) \in \mathcal S_k$, event~$\mathcal A \in \sigma( \mathbf X_s: s \ge 0)$ and stationary measure~$\mu$ for the voter model with dynamic anti-voter bonds, we have that
\begin{equation}\label{eq_2_prof_ergo_limit}
    \lim_{t \to \infty} \int E_{\mathbf x,\mathsf i,A,B}\big[ D(\eta,\zeta,\mathbf X_t,\mathbf i_t,A_t,B_t) \cdot \mathds{1}_\mathcal{A} \big] \; \mu \big(\mathrm{d}(\eta,\zeta)\big) = \frac{1}{2^k} \cdot P_{\mathbf x,\mathbf i,A,B}(\mathcal A).
\end{equation}    
Furthermore, for any~$(x,\mathsf i,A,B) \in \mathcal{S}_1$,
    \begin{equation}\label{eq_3_prof_ergo_limit}
        \mu\big( (\eta,\zeta): \eta(x) = \mathsf i, \zeta \equiv 1 \text{ on } A, \zeta \equiv -1 \text{ on } B \big) = \frac{1}{2} \cdot p^{|A|}(1-p)^{|B|}.
    \end{equation}
\end{corollary}
\begin{proof}
Fix~$(\mathbf x, \mathbf i ,A,B) \in \mathcal S_k$. We denote by~$\E$ the expectation operator~$E_{\mathbf x,\mathbf i,A,B}$. 

By the definition of~$D$ and the triangle inequality, for any~$t > 0$ we have
\begin{align*}
\Big| \int \E& \big[ D(\eta,\mathbf X_t, \mathbf i_t,\zeta,A_t,B_t) \cdot \mathds{1}_\mathcal{A} \big] \mu \big(\mathrm{d}(\eta,\zeta) \big) - \frac{1}{2^k} \cdot \int \E \left[ \varphi_\mathrm{edge}(\zeta,A_t,B_t) \cdot \mathds{1}_\mathcal{A} \right] \mu \big(\mathrm{d}(\eta,\zeta) \big) \Big|\\[2mm]
& \leq \int \E\Big[ \varphi_\mathrm{edge}(\zeta,A_t,B_t) \cdot \mathds{1}_\mathcal{A} \cdot \Big| \E \big[ \varphi_\mathrm{site}(\eta,\mathbf X_t,\mathbf i_t) \mid A_t,B_t, (\mathbf X_s)_{s \ge 0} \big] - \frac{1}{2^k} \Big| \Big] \; \mu \big(\mathrm{d}(\eta,\zeta) \big). 
\end{align*}
By~\eqref{eq_1_prof_ergo_limit_final},  given~$\epsilon > 0$, there exists~$T$ such that for any~$t \ge T$, the right-hand side above is smaller than
\[
\epsilon \cdot \int \E\big[ \varphi_\mathrm{edge}(\zeta,A_t,B_t) \cdot \mathds{1}_\mathcal{A} \big] \; \mu \big(\mathrm{d}(\eta,\zeta) \big).
\]

From the dynamics of the model, the edge marginal of any stationary measure~$\mu$ is~$\pi^\mathrm{edge}_p$. Hence, applying Tonelli's theorem, we have that 
$$\int \E\big[ \varphi_\mathrm{edge}(\zeta,A_t,B_t) \cdot \mathds{1}_\mathcal{A} \big] \; \mu \big(\mathrm{d}(\eta,\zeta) \big) = P_{\mathbf x,\mathbf i,A,B}(\mathcal A) \quad \text{for any } t \ge 0.$$
To obtain~\eqref{eq_3_prof_ergo_limit} we use the stationarity of~$\mu$ together with the duality relation~\eqref{duality_equality}, which together imply that~$\mu( (\eta,\zeta):\eta(x) = \mathsf i, \zeta  \equiv 1  \text{ on } A,  \zeta \equiv -1  \text{ on } B )$ equals
$$ p^{|A|}(1-p)^{|B|}\cdot \int E_{x,\mathsf i,A,B} \big[ D(\eta,Y^1_t, \mathsf i^1_t,\zeta,A_t,B_t) \big] \; \mu(\mathrm{d}(\eta,\zeta)) \quad \text{for any } t \ge 0.$$
We note that~$(Y^1_s,\mathsf i^1_s,A_s,B_s)_{s \ge 0}$ has the same distribution as the independent random walk~\eqref{matrix_q_irw_dyn} for~$k=1$. Therefore, applying~\eqref{eq_2_prof_ergo_limit} we obtain the desired result.
\end{proof}

\subsection{Uniqueness}\label{ss_uniqueness}
In this section, we complete the proof of ergodicity by establishing the uniqueness of the stationary measure~$\mu^\mathrm{dyn}$. 
\begin{lemma}\label{lemma_uniqueness}
The voter model with dynamic anti-voter bonds, defined on an arbitrary graph~$G$, with adoption rates given by a~$Q$-matrix satisfying~\eqref{Q_def_adoption_rw} and with environment parameters~$p \in (0,1)$ and~$\mathsf v \in (0,\infty)$, admits a unique stationary measure.
\end{lemma}
\begin{proof}[Proof of Theorem~\ref{thm_main_ergo}] It readily follows from the previous lemma and Lemma~\ref{lemma_unique_implies_ergo}.
\end{proof}
\begin{proof}[Proof of Lemma~\ref{lemma_uniqueness}]
Let~$\mu_1,\mu_2$ be stationary measures for the process. To obtain that~$\mu_1 = \mu_2$, it is sufficient to show that these measures coincide on sets of the form
$$\{ (\eta,\zeta): \eta(x_1) = \mathsf i_1,\dots,\eta(x_k)=\mathsf i_k, \zeta \equiv +1  \text{ on } A, \zeta \equiv -1  \text{ on } B \} $$
for any~$k \in \N$ and any~$(\mathbf x,\mathbf i,A,B) \in \mathcal S_k$. We proceed by induction on~$k$. Equation~\eqref{eq_3_prof_ergo_limit} of Corollary~\ref{cor_ergo_limit} shows that the statement holds for~$k=1$.

Assuming the statement holds for~$1,\dots,k-1$, it remains to prove its validity for~$k$. To this end, we construct a coupling between independent and coalescing random walks evolving in the same dynamical percolation environment.

We begin by defining on a common probability space~$\Prob$ the process
$$(X^1_t, \dots,X^k_t,\mathsf j^1_t,\dots,\mathsf j^k_t,\mathsf A_t,\mathsf B_t)_{t \ge 0}$$ 
with matrix rate~$Q_\mathrm{irw}(p,\mathsf v,k,Q)$ from~\eqref{matrix_q_irw_dyn}. Our goal is to construct, on the same space, a second process with matrix rate~$Q_\mathrm{crw}(p,\mathsf v,k,Q)$ from~\eqref{matrix_q_crw_dyn}, such that the two processes coincide up to the stopping time
$$ \tau := \inf\{ t \ge 0: |\{X^1_t,\dots,X^k_t\}| = k-1 \},$$
that is, the first time a collision occurs among the independent random walks. 

Let~$m \neq n$ be the indices such that~$X^m_{\tau^-} \neq X^m_{\tau} = X^n_\tau$, that is, at time~$\tau$, walker~$m$ jumps onto the position of walker~$n$. We also consider the process 
$$(\tilde Y^1_s,\dots,\tilde Y^{m-1}_s,\tilde Y^{m+1}_s, \dots \tilde Y^{k}_s, \tilde{\mathsf i}^1_s,\dots, \tilde{\mathsf i}^{m-1}_s, \tilde{\mathsf i}^{m+1}_s,\dots, \tilde{\mathsf i}^k_s,\tilde A_s, \tilde B_s)_{s \ge 0}$$ 
with matrix rate~$Q_\mathrm{crw}(p,\mathsf v,k-1,Q)$ started from
$$(X^1_{\tau}, \dots, X^{m-1}_{\tau}, X^{m+1}_{\tau}, \dots, X^k_{\tau}, \mathsf i^1_{\tau}, \dots, \mathsf i^{m-1}_{\tau}, \mathsf i^{m+1}_{\tau},\dots, \mathsf i^k_{\tau}, \mathsf A_{\tau}, \mathsf B_{\tau}).$$

We now construct the desired process~$(Y^1_t,\dots,Y^k_t,\mathsf i^1_t,\dots, \mathsf i^k_t,A_t,B_t)_{t \ge 0}$ as follows: 
\begin{itemize}
    \item For~$t \in [0,\tau)$, we set
    $$  Y^i_t := X^i_t, \; \mathsf i^i_t := \mathsf j^i_t \; \text{ for } i \in \{ 1,\dots,k\} \quad \text{and} \quad A_t := \mathsf A_t,\; B_t := \mathsf B_t.$$
    \item For~$t \in [\tau,\infty)$, we set
    $$ \left\{ \begin{array}{l}
         Y^i_t := \tilde Y^i_{t - \tau}, \; \mathsf i^i_t := \tilde{\mathsf i}^i_{t - \tau} \; \text{ for } i \in \{1,\dots,k\} \setminus \{m\}; \\[2mm]
         Y^m_t := \tilde Y^n_{t-\tau}, \; \mathsf i^m_t := \mathsf j^m_{\tau} \cdot (\tilde{\mathsf i}^n_0)^{-1} \cdot  \tilde{\mathsf i}^n_{t-\tau}   
    \end{array} \right. \text{ and } \; A_t := \tilde A_{t-\tau}, B_t := \tilde B_{t - \tau}. $$
    In the definition, we use the factor~$\mathsf j^m_\tau \cdot (\tilde{\mathsf i}_0)^{-1}$ to ensure that the sign process~$(\mathsf i^m_t)$ evolves according to the same dynamics as the sign process of~$(Y^n_t)$, but started from~$\mathsf j^m_\tau$. If~$\mathsf j^m_\tau \cdot (\tilde{\mathsf i}_0)^{-1} = 1$ the sign processes have synchronized; see Remark~\ref{rmk_coalescing_synchronization}.
\end{itemize}
Then, it is easy to convince oneself that this process has~$Q_\mathrm{crw}$ as its matrix rate for~$k$ random walks and it agrees with~$(\mathbf X_t,\mathbf j, \mathsf A_t, \mathsf B_t)$ up to time~$\tau$. 

Using the stationarity of~$\mu_1$ and the duality relationship~\eqref{duality_equality}, we have that for any~$(\mathbf x, \mathbf i,A,B) \in \mathcal S_k$ and any time~$t \ge 0$, 
\begin{align*}
\hat{\mu}_1(\mathbf x, \mathbf{i},A,B) :&= \mu_1(\eta(x_1)=\mathsf i_1, \dots, \eta(x_k)=\mathsf i_k, \zeta \equiv +1 \text{ on }  A,\zeta \equiv -1 \text{ on } B ) \\[2mm]
& = E_{\mathbf x, \mathbf i, A,B}\Big[\int D(\eta,\zeta,\mathbf Y_t, \mathbf i_t, A_t, B_t) \; \mu_1\big( \mathrm{d}(\eta,\zeta) \big) \Big] \cdot p^{|A|}(1-p)^{|B|}.    
\end{align*}
The analogous equality holds for the stationary measure~$\mu_2$. Then, for any~$t \ge 0$,
\begin{equation}\label{mu_1_m_2_upper_bound}
\big| \hat{\mu}_1( \mathbf x, \mathbf i,A,B) - \hat{\mu}_2( \mathbf x, \mathbf i,A,B) \big| \cdot p^{-|A|}(1-p)^{-|B|} \leq  |\Psi_1(t)| + |\Psi_2(t)|,
\end{equation}
where
\begin{align*}
& \Psi_1(t) = E_{\mathbf x, \mathbf i, A,B}\Big[ \Big( \int D(\eta,\zeta,\mathbf X_t, \mathbf j_t,\mathsf A_t,\mathsf B_t) \; \mathrm{d}\mu_1 - \int D(\eta,\zeta,\mathbf X_t, \mathbf j_t,\mathsf A_t,\mathsf B_t) \; \mathrm{d} \mu_2 \Big) \cdot \mathds{1}\{t < \tau \} \Big]; \\[2mm]
& \Psi_2(t) = E_{\mathbf x, \mathbf i, A,B}\Big[ \Big( \int D(\eta,\zeta,\mathbf Y_t, \mathbf i_t,A_t,B_t) \; \mathrm{d}\mu_1 - \int D(\eta,\zeta,\mathbf Y_t, \mathbf i_t, A_t, B_t) \; \mathrm{d}\mu_2 \Big) \cdot \mathds{1}\{t \ge \tau \} \Big].
\end{align*}
Observe that~$\{t < \tau\} \in \sigma( \mathbf X_s : s \ge 0 )$. Thus,~\eqref{eq_2_prof_ergo_limit} from Corollary~\ref{cor_ergo_limit} implies that~$\Psi_1(t) \to 0$ as~$t \to \infty$. 

On the other hand, note that on the event~$\{t \ge \tau\}$, we have~$|\{Y^1_t,\dots,Y^k_t\}| \leq k - 1$. Then,
\begin{align*}
   \int D(\eta,\zeta, \mathbf{Y}_t, \mathbf{i}_t,A_t,B_t) \; \mu_1 \big( \mathrm{d}(\eta,\zeta) \big) & \stackrel{\eqref{psi_varphi_def}} = p^{|A_t|}(1-p)^{|B_t|} \cdot \hat{\mu}_1( \mathbf Y_t, \mathbf i_t, A_t,B_t) \\[2mm]
   & \hspace{1mm} \stackrel{ (\ast) } = p^{|A_t|}(1-p)^{|B_t|} \cdot \hat{\mu}_2( \mathbf Y_t, \mathbf i_t, A_t,B_t) \\[2mm]
   & \stackrel{\eqref{psi_varphi_def}} = \int D(\eta,\zeta, \mathbf{Y}_t, \mathbf{i}_t,A_t,B_t) \; \mu_2\big( \mathrm{d}(\eta,\zeta) \big),
\end{align*}
where~$(\ast)$ holds by the inductive step. If the coalescing walkers do not synchronize; see Remark~\ref{rmk_coalescing_synchronization}, then both sides of~$(\ast)$ vanish, so the equality still holds. Hence,~$\Psi_2(t) = 0$. Plugging these into~\eqref{mu_1_m_2_upper_bound} and letting~$t \to \infty$ completes the inductive step and the proof.
\end{proof}

\noindent \textbf{Acknowledgments.} J.A. was partially supported by FAPESP (grants 2023/13453-5 and 2025/02707-1).

\bibliographystyle{alpha}
\bibliography{Ref.bib}

@article{PeresStaufferSteif2015,
  title={Random walks on dynamical percolation: mixing times, mean squared displacement and hitting times},
  author={Peres, Y. and Stauffer, A. and Steif, J.},
  journal={Probab. Theory Relat. Fields},
  volume={162},
  pages={487–530},
  year={2015}
}

@article{Swart2017,
    author = {Swart, J.},
    title = {A Course in Interacting Particle Systems},
    journal = {arXiv preprint arXiv:1703.10007v4},
    year ={2022} 
}

@article{CliffordSudbury73,
    author = {Clifford, P. and Sudbury A.},
    title = {A model for spatial conflict},
    journal = {Biometrika},
    volume = {60},
    pages = {581-588},
    year = {1973}
}

@book{Liggett2005,
  title={Interacting Particle Systems. Grundlehren der Mathematischen Wissenschaften [Fundamental Principles of Mathematical Sciences]},
  author={Liggett, T. M.},
  year={1985},
  publisher={Springer, New York}
}

@book{Liggett2010,
  title={Continuous time Markov processes : an introduction},
  author={Liggett, T. M.},
  year={2010},
  publisher={ American Mathematical Society}
}

@article{HolleyLiggett75,
    author ={Holley, R. and Liggett, T. M.},
    title = {Ergodic theorems for weakly interacting systems and the voter model},
    journal = {Annals of Probability},
    volume = {3},
    pages = {643-663},
    year = {1975}
}

@article{Andresetall,
author = {Andres, S. and Gantert, N. and Schmid, D. and Sousi, P.},
title = {{Biased random walk on dynamical percolation}},
volume = {52},
journal = {The Annals of Probability},
number = {6},
publisher = {Institute of Mathematical Statistics},
pages = {2051-2078},
year = {2024},
doi = {10.1214/23-AOP1679},
URL = {https://doi.org/10.1214/23-AOP1679}
}

@article{MaillardMountford2013,
author = {Maillard, G. and Mountford, T. S.},
journal = {Annales de l'I.H.P. Probabilités et statistiques},
number = {1},
pages = {13-35},
title = {Ergodic behaviour of “signed voter models”},
volume = {49},
year = {2013},
}

@article{Gantertetall2005,
     author = {Gantert, N. and Löwe, M. and Steif, J.},
     title = {The voter model with anti-voter bonds},
     journal = {Annales de l'I.H.P. Probabilit\'es et statistiques},
     pages = {767-780},
     publisher = {Elsevier},
     volume = {41},
     number = {4},
     year = {2005},
     doi = {10.1016/j.anihpb.2004.03.007},
     mrnumber = {2144233},
     zbl = {1070.60087},
     language = {en},
     url = {http://www.numdam.org/articles/10.1016/j.anihpb.2004.03.007/}
}

@article{Matloff1977,
author = {Matloff, N.},
title = {{Ergodicity Conditions for a Dissonant Voting Model}},
volume = {5},
journal = {The Annals of Probability},
number = {3},
publisher = {Institute of Mathematical Statistics},
pages = {371 - 386},
year = {1977},
doi = {10.1214/aop/1176995798},
URL = {https://doi.org/10.1214/aop/1176995798}
}

@article{Saada95,
    author = {Saada, E.},
    title = {Un modèle du votant en milieu aléatoire},
    journal = {Annales de l'I.H.P. Probabilités et statistiques},
    volume = {31},
    page = {263-271},
    year = {1995}
}

@article{Majerek05,
author = {Majerek, D. and Nowak, W. and Ziba, W.},
year = {2005},
month = {01},
pages = {},
title = {Conditional strong law of large number},
volume = {20},
journal = {International Journal of Pure and Applied Mathematics}
}

@article{Astoquillca26,
    author = {Astoquillca, J.},
    title = { {On the Stationary Measures of Two Variants of the Voter Model} },
    journal = {Journal of Theoretical Probability},
    year = {2026},
    volume = {39},
    number = {34},
}

@article{Mountford93, 
    title={A coupling of finite particle systems}, volume={30}, 
    number={1}, 
    journal={Journal of Applied Probability}, 
    author={Mountford, T. S.}, 
    year={1993}, 
    pages={258–262},
}

\end{document}